\documentclass{article}
\usepackage{graphicx} 

\usepackage[a4paper,left=2.5cm,top=3cm,right=2.5cm,bottom=3cm]{geometry}

\usepackage{fancyhdr}
\makeatletter

\newcommand\backmatter{%
  \if@openright
    \cleardoublepage
  \else
    \clearpage
  \fi
  \pagenumbering{alpha}
}
\makeatother

\usepackage{algorithm}
\usepackage[noend]{algpseudocode}
\usepackage{amsfonts}
\usepackage{amsmath}
\usepackage{amssymb}
\usepackage{amstext}
\usepackage{amsthm}
\usepackage{array}
\usepackage{authblk}
\usepackage{booktabs}
\usepackage[colorlinks=true, linkcolor = black, citecolor = dgreen]{hyperref} 
\usepackage{cleveref}
\usepackage[usenames,dvipsnames]{xcolor}
\usepackage{colortbl}
\usepackage{emptypage}
\usepackage[shortlabels]{enumitem}
\usepackage[english]{babel}
\usepackage{eqparbox}
\usepackage{extarrows}
\usepackage{eucal}
\usepackage{float}
\usepackage{graphicx} \graphicspath{ {img/} }
\usepackage{import}
\usepackage{latexsym}
\usepackage{marvosym}
\usepackage{mathrsfs}
\usepackage{mathtools}
\usepackage{multirow}
\usepackage[numbers]{natbib}
\usepackage{overpic}
\usepackage{setspace}
\usepackage{soul}
\usepackage{subfigure}
\usepackage{todonotes}
\usepackage{tikz}
  \usetikzlibrary{matrix,calc}
  \usetikzlibrary{decorations.markings,decorations.pathreplacing}
  \usetikzlibrary{patterns}
\usepackage{url}
\usepackage{verbatim}

\usepackage{xspace}

\setlist[enumerate]{label = (\arabic*), ref = \arabic*}
\crefformat{enumi}{item~(#2#1#3)}
\Crefformat{enumi}{Item~(#2#1#3)}
\crefrangeformat{enumi}{items~(#3#1#4) to~(#5#2#6)}
\Crefrangeformat{enumi}{Items~(#3#1#4) to~(#5#2#6)}

\definecolor{Gray}{gray}{0.9}
\definecolor{dgreen}{rgb}{0,.8,0}
\definecolor{red}{HTML}{D62728}
\definecolor{blue}{RGB}{ 0, 109, 219}

\theoremstyle{plain}
\newtheorem{theorem}{Theorem}[section]
\newtheorem{lemma}[theorem]{Lemma}
\newtheorem{corollary}[theorem]{Corollary}
\newtheorem{proposition}[theorem]{Proposition}
\theoremstyle{definition}
\newtheorem{definition}[theorem]{Definition}

\newtheorem{example}[theorem]{Example}

\theoremstyle{remark}

\newtheorem{notation}[theorem]{Notation}

\newtheorem{construction}[theorem]{Construction}
\newtheorem{observation}[theorem]{Observation}

\newcommand{\bdy}{\partial}
\renewcommand{\setminus}{{\smallsetminus}}

\newcommand{\dualD}{\pi(L)^*}

 {\begin{list}{}%
    {\setlength{\leftmargin}{#1}}%
  \item[]%
 }
{\end{list}}

\tikzset{square matrix/.style={
    matrix of nodes,
    column sep=-\pgflinewidth, row sep=-\pgflinewidth,
    nodes={draw,
      minimum height=15pt,
      anchor=center,
      text width=13pt,
      align=center,
      inner sep=0pt
    },
  },
  square matrix/.default=2cm
}

\colorlet{trans}{blue!50}
\colorlet{trans2}{red!20}

\algnewcommand\algorithmicforeach{\textbf{for each}}
\algdef{S}[FOR]{ForEach}[1]{\algorithmicforeach\ #1\ \algorithmicdo}

\title{Octahedral decomposition of alternating links in thickened surfaces}

\author{Lecheng Su}
\date{}

\begin{document}
\maketitle
\begin{abstract}
    The octahedral decomposition of classical link complements has been considered and utilised by Weeks, Rubinstein, Sakuma etc. It is even more natural to consider the octahedral decomposition of virtual link complements. In this paper, we explore and generalise the octahedral decomposition used in \emph{snappy} to links in thickened surface. Using the decomposition, we prove nonpositive curvature of the complement and essential-ness of edges in the decomposition.
\end{abstract}

\section{Introduction}

An \emph{alternating link} has a diagram for which crossings alternate over and under when following each component of the link. Alternating links can be studied using geometric techniques, e.g., Thurston~\cite{thurston1982three} on the classification of knots in the 3-sphere, and these are particularly effective in the alternating case. For example, in 1984, Menasco~\cite{Menasco:AltLinks} showed that any classical link in the 3-sphere with a prime alternating diagram is either a $(2, q)$-torus knot, or it is hyperbolic.

To study hyperbolicity and other geometric properties for a link complement, we cut it into geometric or combinatorial pieces. The method of decomposing a classical link complement into polyhedral piece, coming from Thurston's work~\cite{Thurston:Notes}, has been generalised by Howie-Purcell~\cite{HowiePurcell} to links with a diagram on an orientable surface in general 3-manifolds. Furthermore, in their 2024 work, Purcell and Su~\cite{purcell2024alternating} generalised the theory to links with a diagram on nonorientable surfaces. They showed that these class of link complements can be obtained by decorating the double cover of the surface and gluing to itself.

There exists another commonly used decomposition method for a link exterior or complement in $S^3$ called the \emph{octahedral decomposition}, where one places one octahedron on each crossing and glues the octahedra together to get back the link complement in $S^2 \times I$, after capping off a pair of 2-spheres with two balls.

In their 1992 paper~\cite{aitchison1992cusp}, Aitchison and Rubinstein study cusp structure of classical alternating links using cubical decomposition coming from the octahedral decomposition, together with the classical polyhedral decomposition. Earlier in the 1980s, in his PhD thesis, Weeks~\cite{weeks1985hyperbolic} utilizes a triangulation that can be derived from octahedral decomposition as a way of triangulating a link complement. He used this method for calculating the hyperbolic structure of link complements in $S^3$. It was implemented in the famous 3-manifold geometry software \emph{SnapPea} based on the work of Weeks in his thesis, now maintained as \emph{Snappy}~\cite{culler2016snappy} in Python language.

More recently, in their 2018 work~\cite{sakuma2018application}, Sakuma and Yokota carefully described the octahedral decomposition and utilized it to obtain a cubical decomposition of alternating link exterior. They managed to obtain non-positive curvature and volume result for alternating link exterior. Also in their 2018 work~\cite{kim2018octahedral}, Kim-Kim-Yoon utilized the same octahedral decomposition and put labels on the octahedra to derive a system of pseudo-hyperbolic developing maps. They study a canonical representation for alternating links and Ptolemy equations in the 2018 work and its sequel~\cite{kim2023octahedral}.

In his 2018 work~\cite{sakai2018characterization}, Sakai used the cubical decomposition from Sakuma-Yokota's work to characterize alternating link exterior using checkerboard surfaces, obtaining a result similar to Howie~\cite{howie2017characterisation} and Greene~\cite{MR3694566}, but in the language of cube complexes.

Our methods here generalise those of~\cite{kim2018octahedral}~\cite{sakuma2018application}~\cite{sakai2018characterization}. The idea is to associate the link exterior in thickened surfaces $S\times I$ or $S\widetilde{\times} I$ with a cube complex, depending on the orientability of $S$, which has almost the same topology as the exterior, except for a few vertices coming from coning the boundary surface $\widetilde{S}$.

We will prove the following theorem about the geometry of alternating links in thickened orientable surfaces:
\begin{theorem}\label{Thm:npc}
	Let $L$ be link in $S\times I$ or $S\widetilde{\times} I$ with a reduced alternating diagram $D$ on $S$, such that $S\times I \setminus L$ or $S\widetilde{\times} I$ is hyperbolic. Let $\mathcal{C}$ be the cube complex associated to the octahedral decomposition. Suppose further that $e(D,S)\geq 4$. Then
	\begin{enumerate}
		\item The cube complex $\mathcal{C}$ is nonpositively curved.
		\item The edges of $\mathcal{C}$ are essential.
	\end{enumerate}
\end{theorem}

This is a similar result obtained by Sakuma-Yokota in their 2018 work~\cite{sakuma2018application}. However, instead of considering classical knots, we look at links in thickened surfaces. In the classical setting, Weeks~\cite{weeks2005computation} has to drill two tubes to connect the link diagram to the two balls since they also start with the thickened 2-sphere $S^2\times I$ and they want to recover the link complement in $S^3$. It is not the case for links in thickened surfaces, where we don't need to do anything since we start with considering the thickened surface. 

\emph{Essential edges} of a polyhedral decomposition or triangulation can lead to solutions to a hyperbolic structure or representation, since edges of this kind will lift to geodesics in hyperbolic 3-space $\mathbb{H}^3$ if a knot complement is hyperbolic. People have been trying to obtain essential triangulations using different methods in 3-manifold topology, for example, Hodgson-Rubinstein-Segerman-Tillmann~\cite{hodgson2015triangulations} utilize one-vertex triangulation, Heegaard splitting etc to obtain an essential or strongly essential triangulation. Sakuma-Yokota~\cite{sakuma2018application} also utilized the essential triangulation coming from the octahedral decomposition to give an explicit solution to hyperbolic equations. Note that a similar result will be given in Subsection~\ref{subsec:nonoct}, which indicates that an alternating link complement in a thickened nonorientable surface also admits a nonpositively curved cubing.

Nonpositive curvature is an important aspect in geometric group theory, especially after the proof of virtual Haken conjecture using nonpositively curved cube complexes. Wise~\cite{wise2006subgroup} has studied the relationship between prime alternating diagrams on $S^2$ and the nonpositive curvature of the classical two-dimensional \emph{Dehn complex}. Some interesting results have also been published in the literature on the two-dimensional nonpositive curvature phenomena of the link complement in thickened surfaces. The coned Dehn complex is well studied by Harlander~\cite{harlander2003generalized}, yet they construct the Dehn complex differently. They directly consider a deformation retract of the link complement to obtain a ``2-skeleton" that captures the information for the fundamental group $\pi_1(S\times I\setminus L)$. Their retraction from a link in thickened surface is shown in Figure~\ref{Fig:harlander}. In their retraction, a basic block is coming from a crossing in the thickened surface. It is not obvious to show that their definition of Dehn complex gives the same as the classical defintion. However it will not be shown in this paper.

\begin{figure}
	\centering
	\includegraphics[width=0.3\linewidth]{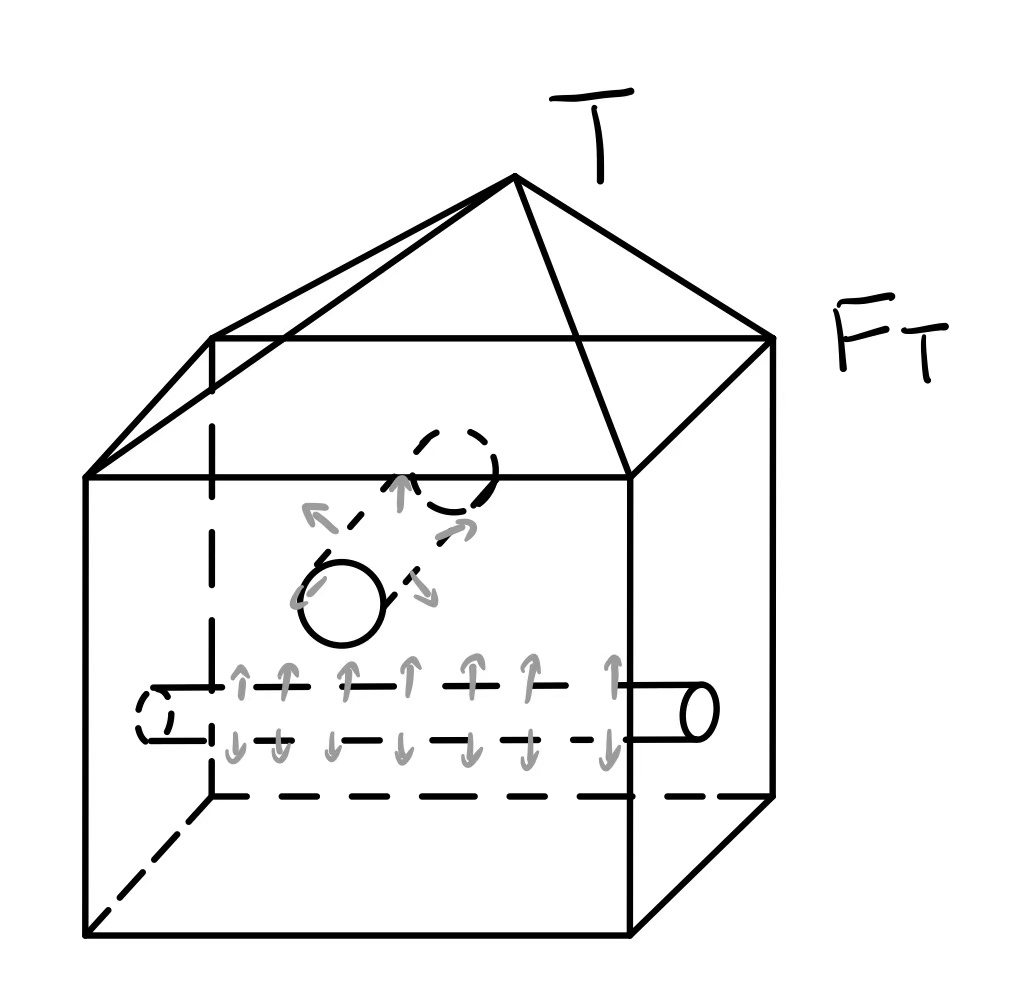}
	\caption{The retraction process from Harlander's work to obtain the 2-complex}
	\label{Fig:harlander}
\end{figure}

\subsection{Virtual knots and links}
In their 1999 work, Kauffman~\cite{Kauffman:Virtual} introduced the new class of knots called \emph{virtual knots}, where in addition to classical crossings which can be put on different hemispheres of a crossing ball~\cite{Menasco:AltLinks}, they have one extra kind of crossing called \emph{virtual crossings}. The class of virtual crossings are not really crossings. They can rather be viewed as two strands running through different parts of a handle, forming a “crossing” when viewed from outside of or far away from a handle.

Except from a planar, diagrammatic point of view, virtual knots can also be thought of as knots in thickened closed orientable surfaces, where one attaches a handle to each virtual crossing. However, this is not the optimal way of building the \emph{projection surface}. In their 2002 work~\cite{kamada2000abstract}, Carter, Kamada and Saito introduced the following notion so that we have an equivalence relation for virtual knots:

\begin{definition}[stable equivalence]
    Two virtual knots are \emph{stably equivalent} if their underlying projection surfaces are related by deletions and additions of one-handles that do not contain any portion of the diagram, which are called \emph{stabilisations} and \emph{destabilisations}.
\end{definition}

\subsection{Kuperberg's theorem}
In early 2000s, it still remains a question whether we get the same virtual knot after the (de)stabilisations or not. In 2003, this problem was solved by Kuperberg's annulus theorem:
\begin{theorem}[\cite{kuperberg2003virtual}, Theorem~1]
	Every stable equivalence class of links in thickened surfaces has
	a unique irreducible representative.
\end{theorem}

This theorem says that we will always arrive at the same \emph{minimal genus projection surface} if we follow Kuperberg's construction. From this point, we want to restrict our view on the minimal genus surface, since by Kuperberg's theorem, it is always the correct surface for us to look at. We will view the link as embedded in the thickened surface $S\times I$ with $S$ being the ``correct" surface. Moreover, we extend this point of view by looking at knots and links in $S\times I$ or $S\widetilde{\times}I$, depending on the orientability of $S$, with cellular diagrams on $S$.

\subsection{Organization of the paper}
In Section~\ref{sec:Def}, we carefully define the octahedral decomposition of a virtual link complement. In Section~\ref{sec:cube}, we transform the octahedral decomposition to a cubical one through a coning construction, and we prove the main theorem for orientable case $S\times I$. In Section~\ref{sec:tet}, we show that by cutting the octahedra further into tetrahedra, we can relate the octahedral decomposition to the stellated triangulation used by Futer and Kaplan-Kelly~\cite{FKP:quasifuchsian}. Later in the section, we carefully write down the octahedral decomposition for links in $S\widetilde{\times}I$ and show the same nonpositive curvature result for links in $S\widetilde{\times} I$. Finally, in Section~\ref{sec:open}, we elaborate some open questions relating to the octahedral decomposition of virtual link complements.

\section{Octahedral decomposition of virtual link complement}\label{sec:Def}

Let $L$ be a link in $S\times I$ with a weakly prime, connected, reduced alternating diagram on $S\times \{0\} $ in the sense of Howie-Purcell~\cite{HowiePurcell}, where $I=[-1,1]$. In Sakuma-Yokota~\cite{sakuma2018application}, they proved that the cubing obtained from the octahedral decomposition of a prime alternating link in $S^3$ with a reduced alternating diagram is nonpositively-curved and the edges of the decomposition are essential. We generalise their case to $S\times I\setminus L$.

We start by setting up the following conventions and notations. Let $S=S_g$, then $\chi(S)=2-2g$. Let $\pi: S\times I\to S$ be the projection and $\pi(L)$ be the diagram on $S$, which is a 4-valent graph on $S\times \{0\}$. Let $\pi(L)^*$ be the dual graph embedded in $S\times \{0\}$, such that $\pi(L)\cap\pi(L)^*=P_1, P_2, \ldots, P_{2c}$. We call these the middle points of $L$. The vertices of $\pi(L)$ are denoted $X_1, X_2, \ldots, X_{c}$, as crossings of the link diagram. The vertices of $\dualD$ are denoted $R_1, R_2, \ldots, R_{c+\chi}$, corresponding to the regions of the link diagram. Observe that the number of regions and the number of vertices are as claimed by an Euler characteristic argument, where $\chi=\chi (S)=2-2g$. See Figure~\ref{Fig:dual} for our setting.

\begin{figure}
	\centering
	\begin{overpic}{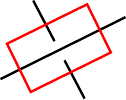}
		\put(45,50){$X_i$}
		\put(90,40){$R_j$}
	\end{overpic}
	\caption{Notation on the diagram and its dual}
	\label{Fig:dual}
\end{figure}

The strands of $L$ connecting one middle point to another are called overstrands and understrands, depending on which half of $S\times [-1,1]$ they lie in.

Sakuma and Sakai~\cite{sakai2024two} have written a nice description of this decomposition, and they use it and its interaction with the checkerboard complex to rewrite a proof of a Kleinian group-theoretic theorem for classical hyperbolic 2-bridge links~\cite{agol2002classification}. In Sakuma-Sakai's proof, they utilized the $CAT(0)$-space tiled by cubes, instead of hyperbolic 3-space tiled by hyperbolic polyhedra.

We can look at the link exterior and the link complement separately. The link exterior $S\times I \setminus N(L)$ can be obtained from decomposition of the link complement $S\times I\setminus L$ by truncating certain vertices corresponding to the link $L$. Or the link complement can be obtained by attaching a series of collar neighborhood of the boundary of the link exterior coming from the link, see Sakai-Sakuma~\cite{sakai2024two} for the detailed construction.

The following construction is a generalised version of the \emph{Aitchison complex}, which we loosely refer to as the octahedral decomposition used by Aitchison and Rubinstein~\cite{aitchison1992cusp} in the 90s:

\begin{construction}[generalized octahedral decomposition]\label{const: octa_decomp}
	The link complement can be obtained as follows.
	\begin{enumerate}
		\item The link exterior in the thickened surface $S\times I \setminus N(L)$ is decomposed into $c$ truncated octahedra, one for each crossing. Here we place the octahedra with two opposite vertices on the link one at the over-crossing and one at the under. These are truncated because we are taking the link exterior. They are the top and bottom vertices in Figure~\ref{Fig:octdecomp}. The remaining four vertices are also truncated. Initially, take them to lie on the surface $S\times \{0\}$. We call them \emph{middle vertices}. Note that in the figure, the middle level truncation is not shown, and the reason will be shown in later context.
		
		\begin{figure}
			\centering
			\includegraphics[width=0.2\linewidth]{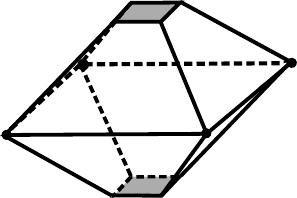}
			\caption{What each octahedron looks like in the decomposition for $S\times I$. The gray part tiles the boundary of the regular neighborhood of the link in $M$. Neighborhoods of the 4 horizontal vertices lie on the coned surfaces $S\times \{ \pm 1 \}$}
			\label{Fig:octdecomp}
		\end{figure}
		
		\item We drag two opposite vertices of the 4 middle vertices of each octahedron vertically down to $S\times \{-1\}$. The other two we pull up to $S\times \{1\}$. See Figure~\ref{Fig:decomp}
		
		\begin{figure}
			\centering
			\begin{overpic}[width=0.3\linewidth]{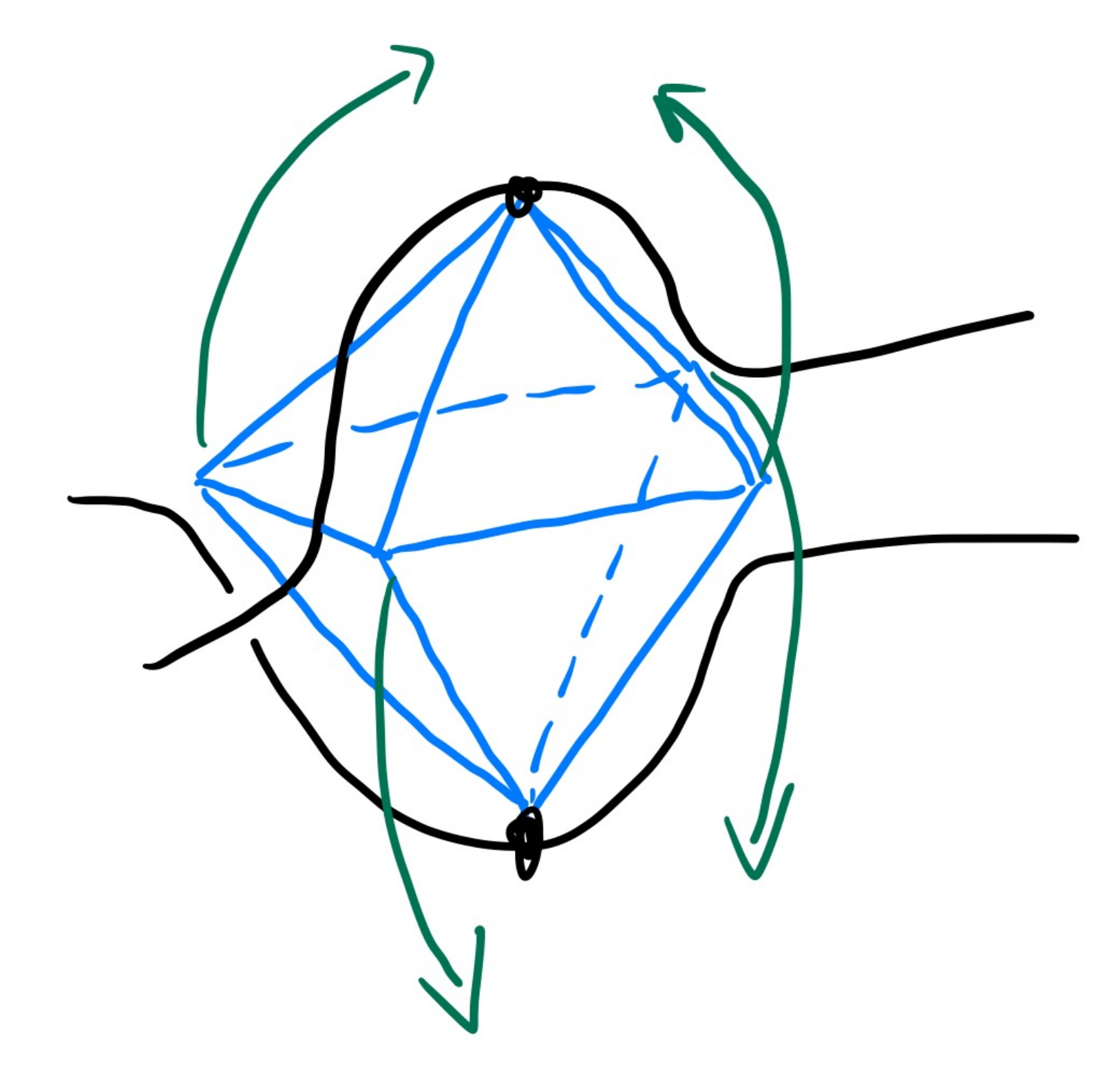}
				\put(40,88){$S\times \{+1\}$}
				\put(45,5){$S\times \{-1\}$}
			\end{overpic}
			\caption{Decomposition of link complement into octahedra according to the diagram.}
			\label{Fig:decomp}
		\end{figure}
		
		\item Each triangular face of an octahedron is glued to a unique face of another octahedron from a neighboring crossing. See Figure~\ref{fig:gluing} for the specific gluing, where the faces shaded red are glued together through the face $L v_+ v_-$ in the middle.
		
		\begin{figure}
			\centering
			\begin{overpic}[width=0.7\linewidth]{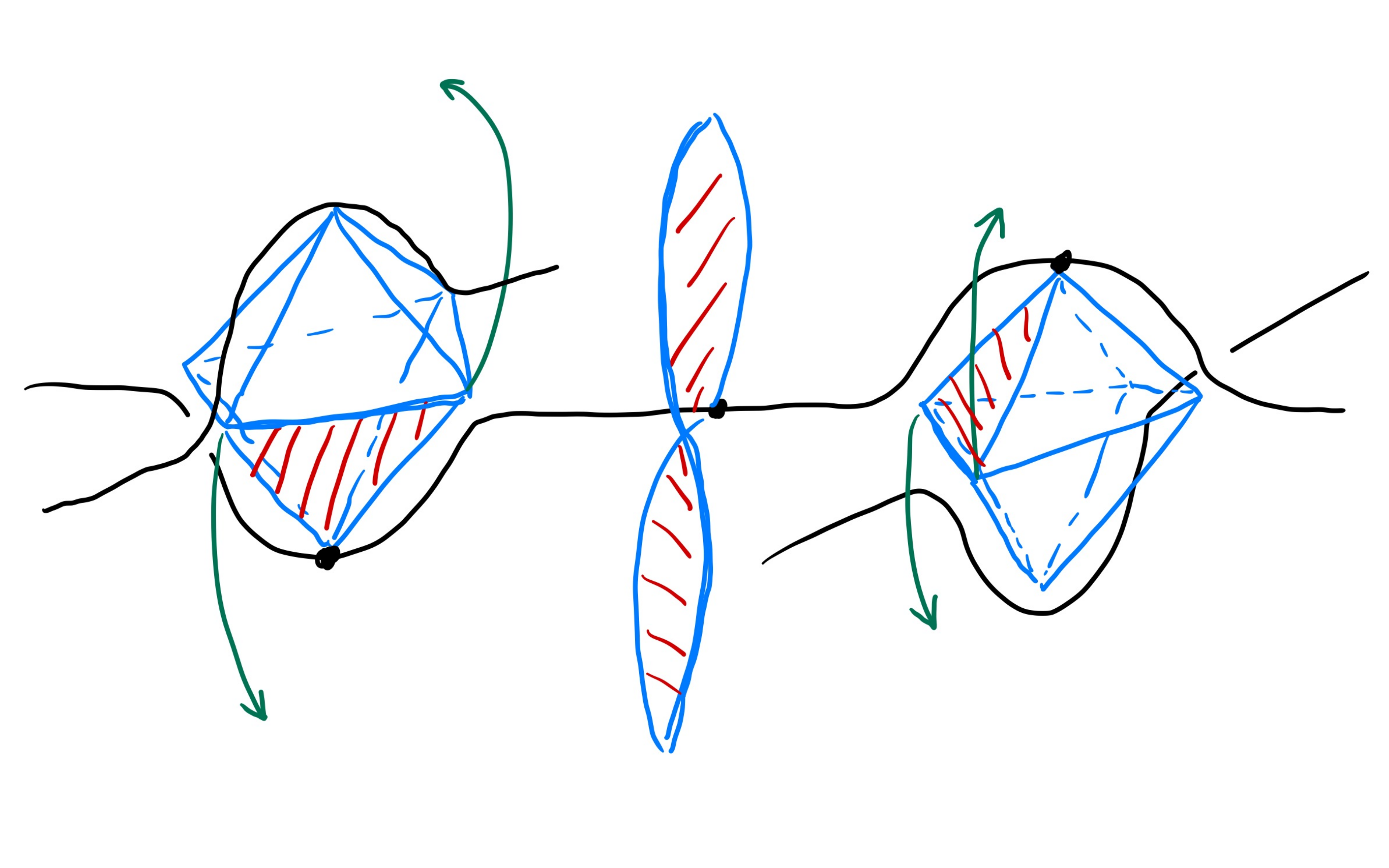}
				\put(50,57){$v_+$}
				\put(50,5){$v_-$}
                \put(25,34){\textcolor{blue}{$R$}}
                \put(66,28){\textcolor{blue}{$R$}}
                \put(43,38){\textcolor{blue}{$R$}}
                \put(24,20){$L$}
                \put(53,34){$L$}
                \put(78,46){$L$}
			\end{overpic}
			\caption{Gluing neighboring octahedra together along a long triangular face}
			\label{fig:gluing}
		\end{figure}
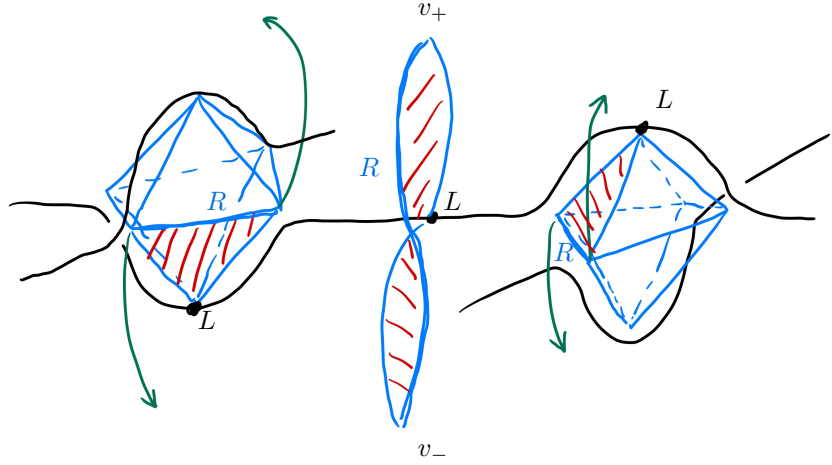

		\item All faces can be glued this way, and after gluing the $c$ octahedra, the entire atmosphere of the link is covered. In particular, before pulling vertices to $S\times \{\pm 1\}$, each $n-$sided region is covered by $n\times 1/4$ of the middle squares of octahedra. Pulling the vertices to $S\times \{\pm 1\}$ gives an edge running through the region from $S\times \{-1\}$ to $S\times \{1\}$, and the $n$ octahedra all glue to this edge.
		
		\item The checkerboard surfaces are ``hyperplanes" inside each octahedron. See Figure~\ref{Fig:oct}, where blue and green are parts of separate checkerboard surfaces, and they intersect at the red crossing arc of the octahedron. Later in this section we'll see that they do form a hyperplane in the cube complex coming from the octahedral decomposition.
		\begin{figure}
			\centering
			\includegraphics{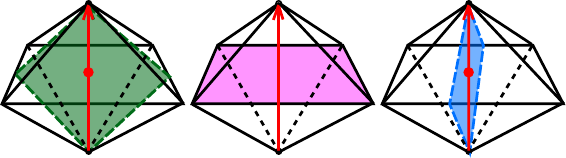}
			\caption{Three codimension-1 objects are inside a single octahedron, with the crossing arc in the middle. The green and blue denote part of the checkerboard surfaces, and the purple plane in the middle denotes the generalised Dehn complex.}
			\label{Fig:oct}
		\end{figure}

		\item The \emph{coned Dehn complex} is made of the ``middle" square of each octahedron, which is colored purple in Figure~\ref{Fig:oct}. Observe that the Dehn complex is dual to the crossing arc.
		
		\item The link of vertices in the octahedral decomposition is divided into two types, one coming from the ideal vertices corresponding to the link $L$, and the other coming from the surfaces $S\times \{\pm 1\}$. The links are shown in Figure~\ref{Fig:link_oct}.
		
		\begin{figure}
			\centering
			\includegraphics[width=0.4\linewidth]{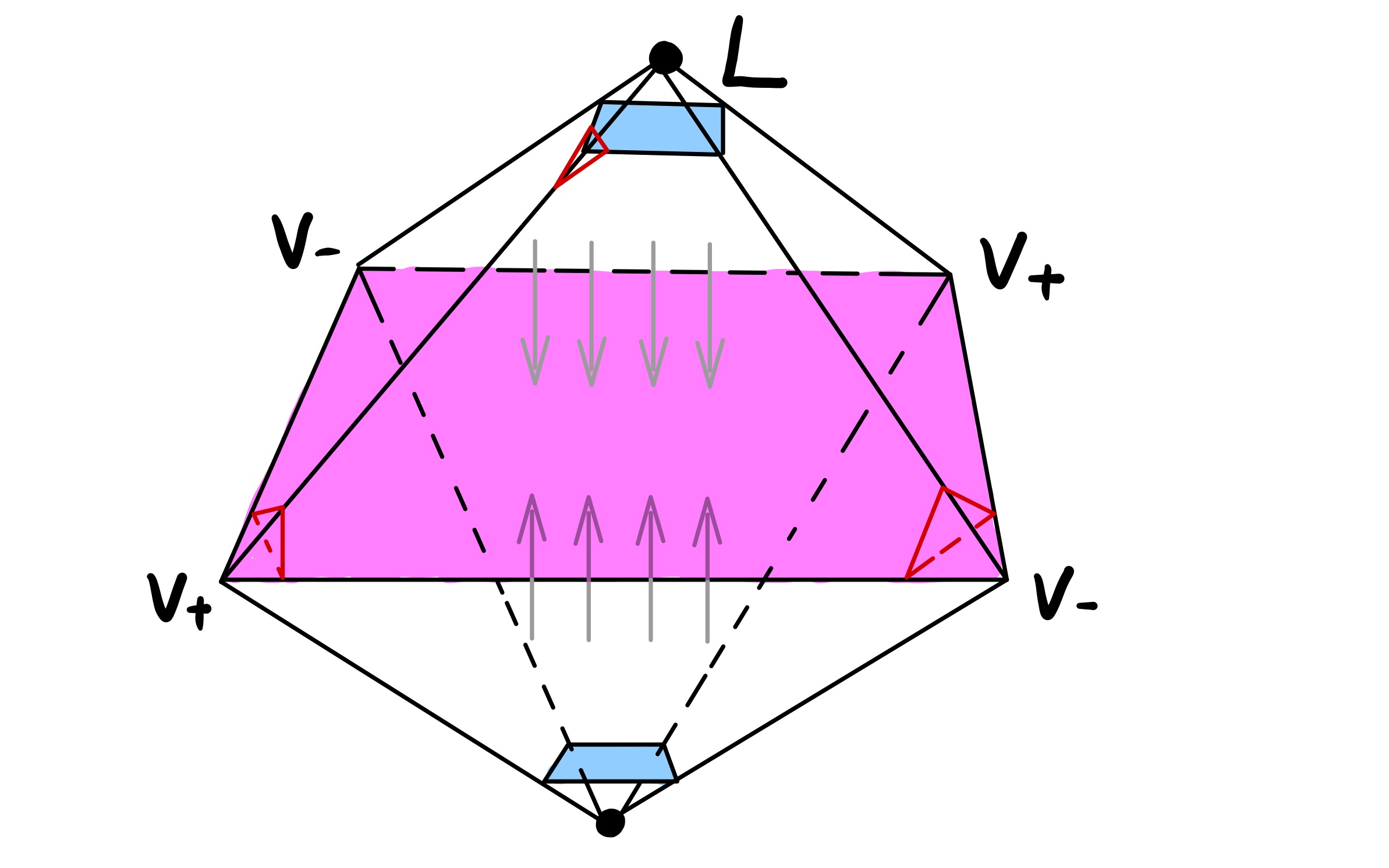}
			\caption{Links of various types of vertices in the octahedron are shown in red. The boundary of the link is shown in blue. The gray arrows denote the retraction of the cube complex to generalised Dehn complex.}
			\label{Fig:link_oct}
		\end{figure}
		
	\end{enumerate}
	
\end{construction}

Note that when we pull the middle vertices to $S\times \{\pm 1\}$, we also identify corresponding edges. We label the vertices with capital letters. We identify the edges to recover the link exterior. See Figure~\ref{fig:lift}, where we first identify the upper edges $AC$ and $AE$ together, and the lower edges $FB$ and $FD$.

\begin{figure}
	\centering
	\includegraphics{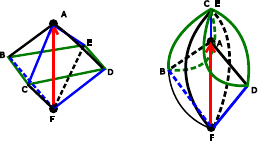}
	\caption{Lift the middle four vertices to the level of boundary surfaces}
	\label{fig:lift}
\end{figure}

After gluing, we get back the link exterior $S\times I \setminus N(L)$. The truncated middle vertices tile the boundary surface.

Note that in other literature like Weeks~\cite{weeks2005computation} or Sakuma-Yokota~\cite{sakuma2018application}, a backward process is written carefully. They start from four ``prisms" around each crossing and deform the prisms so that each prism becomes a tetrahedron and every four tetrahedra around a crossing glue to become an octahedron. See Figure~\ref{fig:deform} for the generalised version of this deformation, where each ideal tetrahedron in the classical setting becomes a partially truncated ideal tetrahedron with two of the vertices truncated to tile the top and bottom surface in the virtual setting.

\begin{figure}
	\centering
	\includegraphics{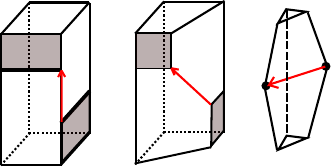}
	\caption{Sakuma-Yokota's construction from spanned prisms}
	\label{fig:deform}
\end{figure}

\section{From octahedra to cubes}\label{sec:cube}

Cube complexes are nice locally-Euclidean, combinatorial objects that are well studied in geometric group theory, for example in Haglund and Wise's 2008 work~\cite{haglund2008special}. They were used to prove virtual Haken conjecture by Agol~\cite{MR3104553}. In 2012, Harlander~\cite{harlander2012hyperbolic} showed that the two dimensional skeleton of the Dehn space associated to a prime, alternating virtual knot is a 2-d nonpositively curved cube complex, and from there, they prove that the group associated to the 2-skeleton is Gromov-hyperbolic.

We think of \emph{cubes} to be Cartesian products $I^n=I\times \ldots \times I$ of intervals $I=[-1,1]$ before we put any geometric structure on them. Our definitions for cube complexes and related items mainly follow textbooks of Wise~\cite{MR2986461} and Bridson-Haefliger~\cite{bridson2013metric}.
\begin{definition}[cube complex]
	A \emph{cube complex} is a cell complex where all the cells are cubes, meaning Cartesian product of intervals $I^n$.
\end{definition}
\begin{definition}[link of a vertex]
	The \emph{link} of a vertex $v$ in a cube complex $\mathcal{C}$ is the simplicial complex where $n$-simplices are corners of $(n+1)$-cubes adjacent with $v$, usually denoted $lk(v)$ or $Link(v)$. 
\end{definition}

The following few definitions also follow from Wise's textbook~\cite{MR2986461}.
\begin{definition}[midcube]\label{Def:midcube}
	A \emph{midcube} in a cube $I^n$ is a subspace of $I^n$ obtained by restricting one of the intervals $I=[-1,1]$ to 0.
\end{definition}

\begin{definition}[hyperplane]\label{Def:hyperplane}
	A \emph{hyperplane} in a cube complex $\mathcal{C}$ is the maximal span of midcubes inside $\mathcal{C}$.
\end{definition}

\begin{example}
	The space obtained by gluing three squares together at a corner is a cube complex, see Figure~\ref{Fig:cubeex} for this space and its hyperplane and vertex link. We'll see from later definitions
\end{example}

\begin{figure}
	\centering
	\includegraphics[width=0.2\linewidth]{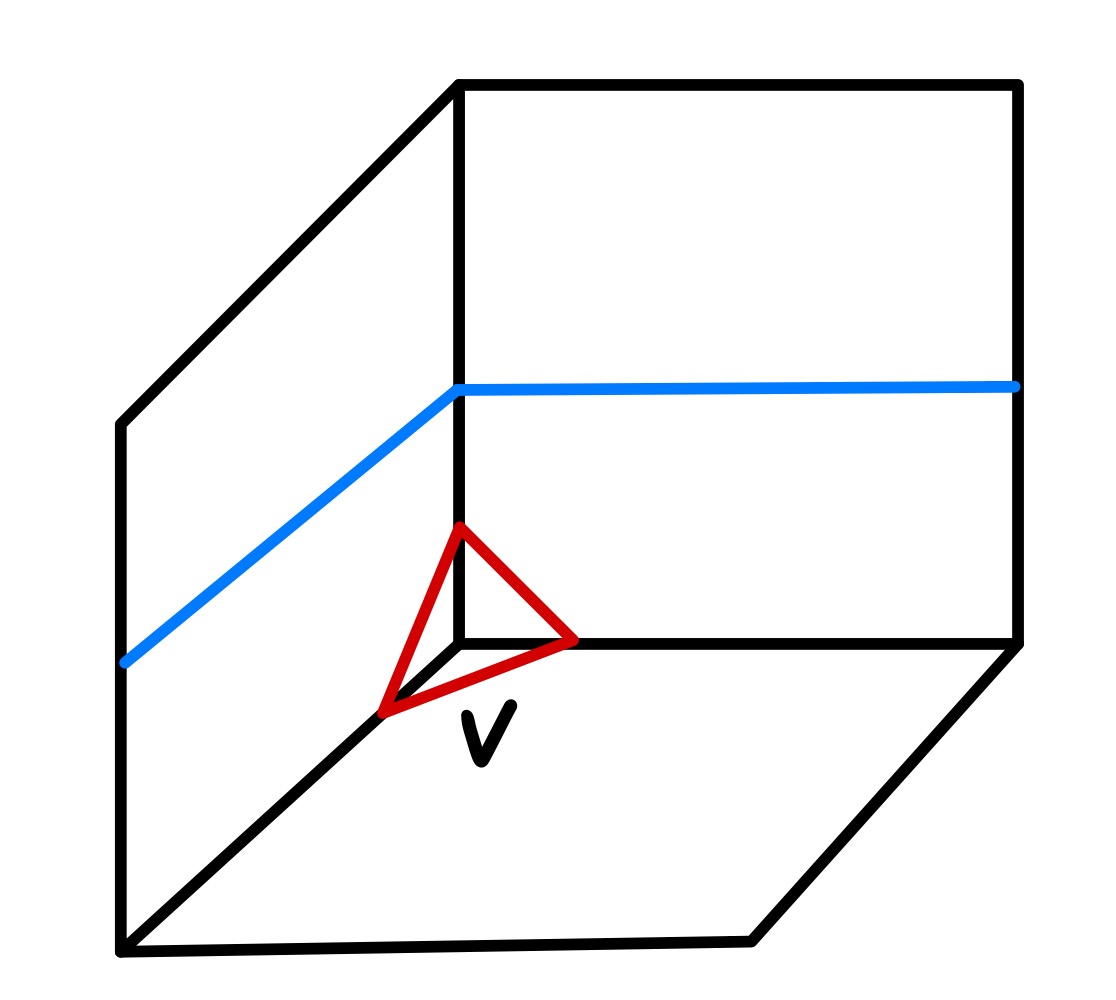}
	\caption{An example of a cube complex is shown in black made up of three squares. The blue line is one hyperplane and the red triangle is the link of the vertex $v$, which shows that the cube complex is not nonpositively curved.}
	\label{Fig:cubeex}
\end{figure}

\begin{example}[checkerboard hyperplane]
	The checkerboard surfaces in the classical octahedral decomposition or the Aitchison complex are both hyperplanes in the cubical decomposition $\mathcal{C}$. See Sakai-Sakuma~\cite{sakai2024two} for more information about classical checkerboard hyperplanes.
\end{example}

\begin{definition}[flag complex]\label{Def:flag}
	A \emph{flag complex} is a simplicial complex such that $n+1$ vertices span an $n$-simplex if and only if they are pairwise adjacent.
\end{definition}
Note that a flag complex is completely determined by a simplicial graph. A graph $\Gamma$ is a flag if and only if $girth(\Gamma)\geq 4$.

The following combinatorial theorem is a useful criterion to judge if a given cube complex has nonpositive curvature without checking the actual angles around each vertex:
\begin{proposition}[Gromov's link condition]\cite[5.20]{bridson2013metric}\label{prop:flag}
	A cube complex is nonpositively curved if for any vertex $x$, the link $Link(x)$ is a simplicial complex which is a flag.
\end{proposition}

We can quotient the surface boundaries $S\times \{\pm 1\}$ to vertices $P_+$, $P_-$ and take the link exterior in the coned manifold $M^*$(truncating the neighborhood of the link in the octahedra). The resulting space $M^*\setminus N(L)$ admits a decomposition into partially truncated octahedra, with truncated vertices on the top and bottom of each octahedron lying on the neighborhood of the link. The truncated vertices that were on surface boundaries $S\times \{\pm 1\}$ before taking the quotient to $M^*$ now become material vertices after the quotient. See Figure~\ref{Fig:octdecomp} again; this is why we did not draw the middle 4 vertices as truncated in the decomposition. 

See Figure~\ref{Fig:cubeoct} for one cube inside an octahedron in the octahedral decomposition of the link exterior after taking the quotient. The exterior $M^* \setminus N(L)$ now divides into cubes.

\begin{figure}
	\centering
	\includegraphics[width=0.2\linewidth]{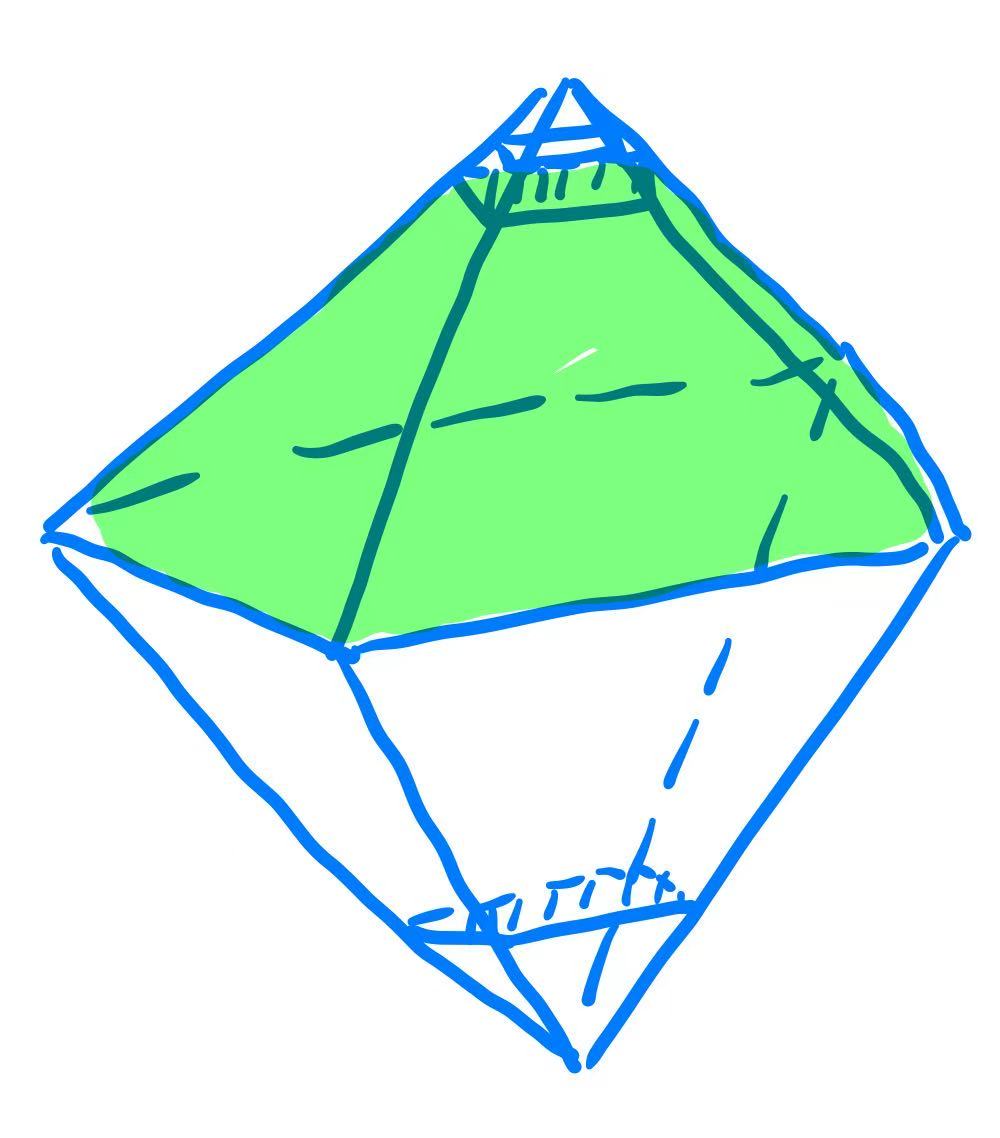}
	\caption{The green part shows a cube inside an octahedron}
	\label{Fig:cubeoct}
\end{figure}

If we further cut the octahedra along the projection surface level, we obtain a cubical decomposition $\mathcal{C}$ of $M^*\setminus N(L)$. The decomposition $\mathcal{C}$ is a cube complex made of only 3-dimensional cubes.

Note that if we retract the cube complex to the middle squares of the octahedra, we get the Dehn complex in the classical case, and this can be generalised naturally, see Sakuma-Yokota~\cite{sakuma2018application} for the classical construction, where they realize that a retraction of their cube complex is isomorphic to the Dehn complex.

\subsubsection{Classical Dehn complex}
The Dehn complex is a 2-complex that is constructed from a link diagram and is a deformation retraction of the link complement. It is related to the Wirtinger complex used to determine the Wirtinger presentation of a classical knot group, see for example, Rolfsen~\cite{rolfsen2003knots}, where there is one single vertex as the base point for calculating the fundamental group, one loop for each over-strand of the diagram as the generators, and one 4-letter relator for each crossing as the relators.

Following Bridson-Haefliger~\cite{bridson2013metric}, the construction of a Dehn complex from a classical knot diagram $\Pi$ that is regular 4-valent goes as follows:

\begin{construction}[classical Dehn complex]\label{const:Dehn}
	Let $\Pi$ be a link diagram on $S^2$ that is checkerboard colored, then the \emph{Dehn complex} $X$ associated to $\Pi$ contains the following:
	\begin{enumerate}
		\item The 2-complex $X$ has exactly two vertices $v_t$ and $v_b$, namely top and bottom vertices.
		\item The 1-skeleton of $X$ contains edges running between top and bottom vertices, corresponding to the regions of the link diagram $\Pi$, one for each region. The edge in the black region is oriented from $v_t$ to $v_b$, and the edge in the white region from $v_b$ to $v_t$. See Figure~\ref{Fig:dehn} for the 1-skeleton. 
		\item For each crossing in $\Pi$, we have a 2-cell that is a square, attached along the edges traveling up and down the regions around each crossing. See Figure~\ref{Fig:relatorsdehn} for the 2-cells corresponding to the figure-8 knot diagram in Figure~\ref{Fig:dehn}.
	\end{enumerate}
\end{construction}

\begin{figure}
	\centering
	\includegraphics[width=0.4\linewidth]{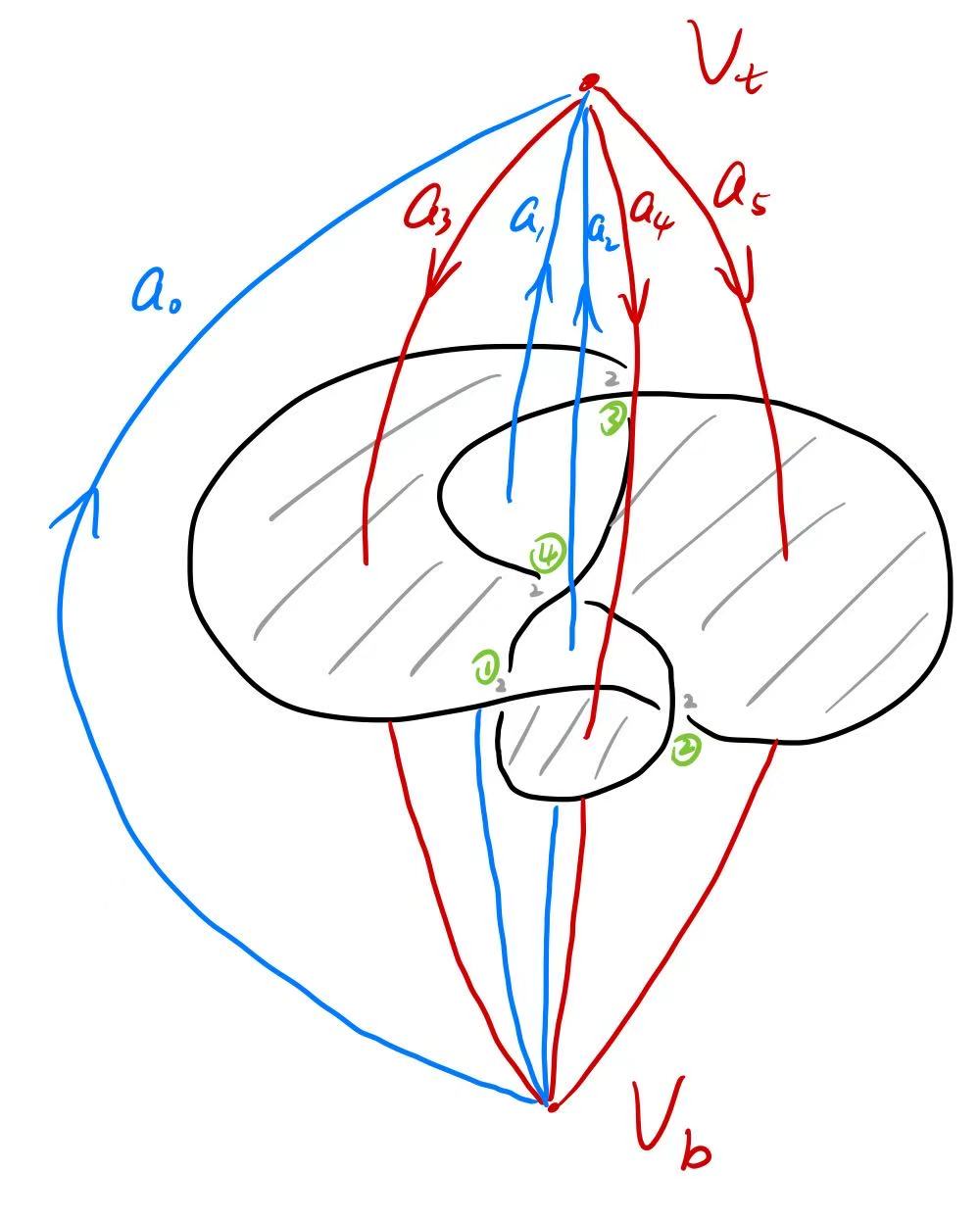}
	\caption{The classical construction of a Dehn complex}
	\label{Fig:dehn}
\end{figure}

\begin{figure}
	\centering
	\includegraphics[width=0.4\linewidth]{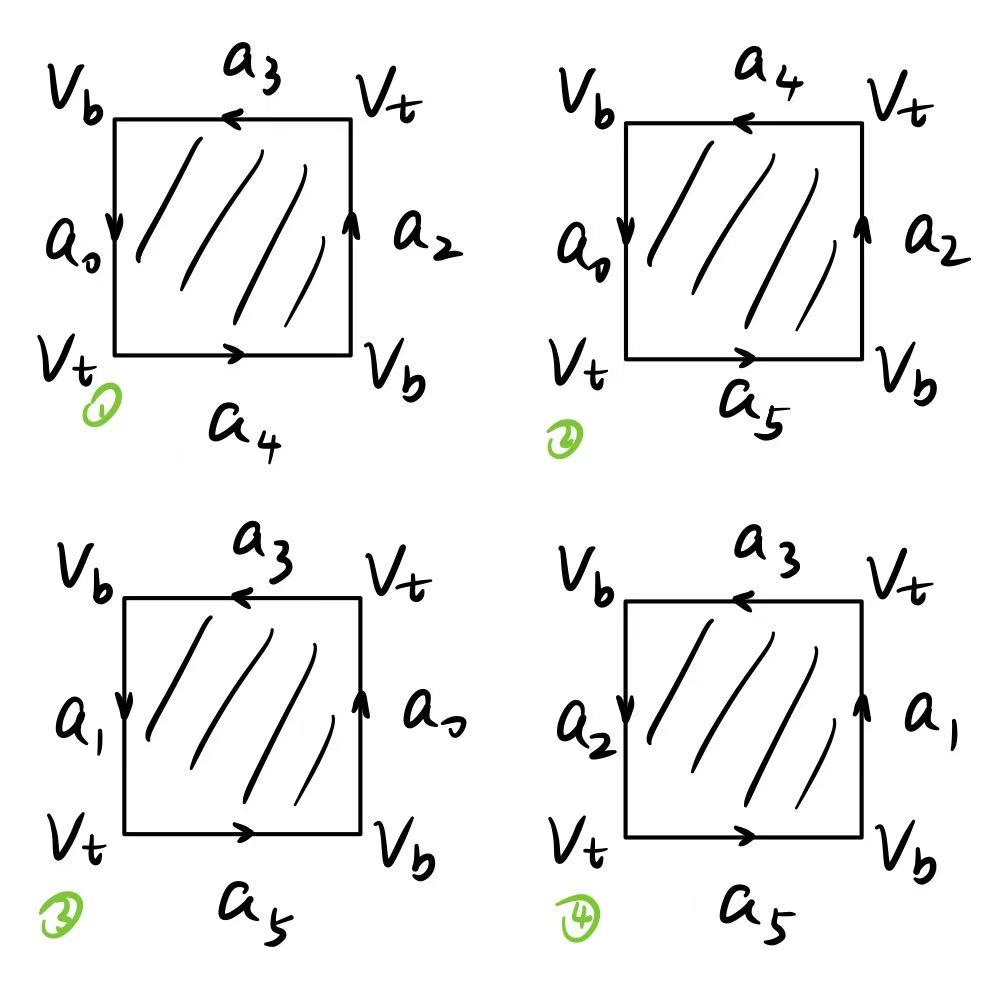}
	\caption{The relators of a Dehn complex}
	\label{Fig:relatorsdehn}
\end{figure}

Using the classical Dehn complex, Wise was able to extract subgroup separability of the figure-8 knot group. Moreover, using the combinatorial property, he was able to obtain the following theorem about alternating diagrams:
\begin{theorem}(Wise~\cite{wise2006subgroup})
	For a knot or link diagram $\pi(L)$, its associated Dehn complex is nonpositively curved if and only if $\pi(L)$ is prime and alternating.
\end{theorem}

\begin{example}[nonpositive curvature of Dehn complex]
	Figure~\ref{Fig:linkfig8} shows that $link(x_-)$ in the Dehn complex can be read off directly from the diagram. The red part shows the link embedded in the figure-eight knot diagram. It is concluded that the Dehn complex associated to figure-eight knot is indeed nonpositively curved, since the link of both vertices are flag complexes. 
\end{example}

\begin{figure}
	\centering
	\includegraphics[width=0.5\linewidth]{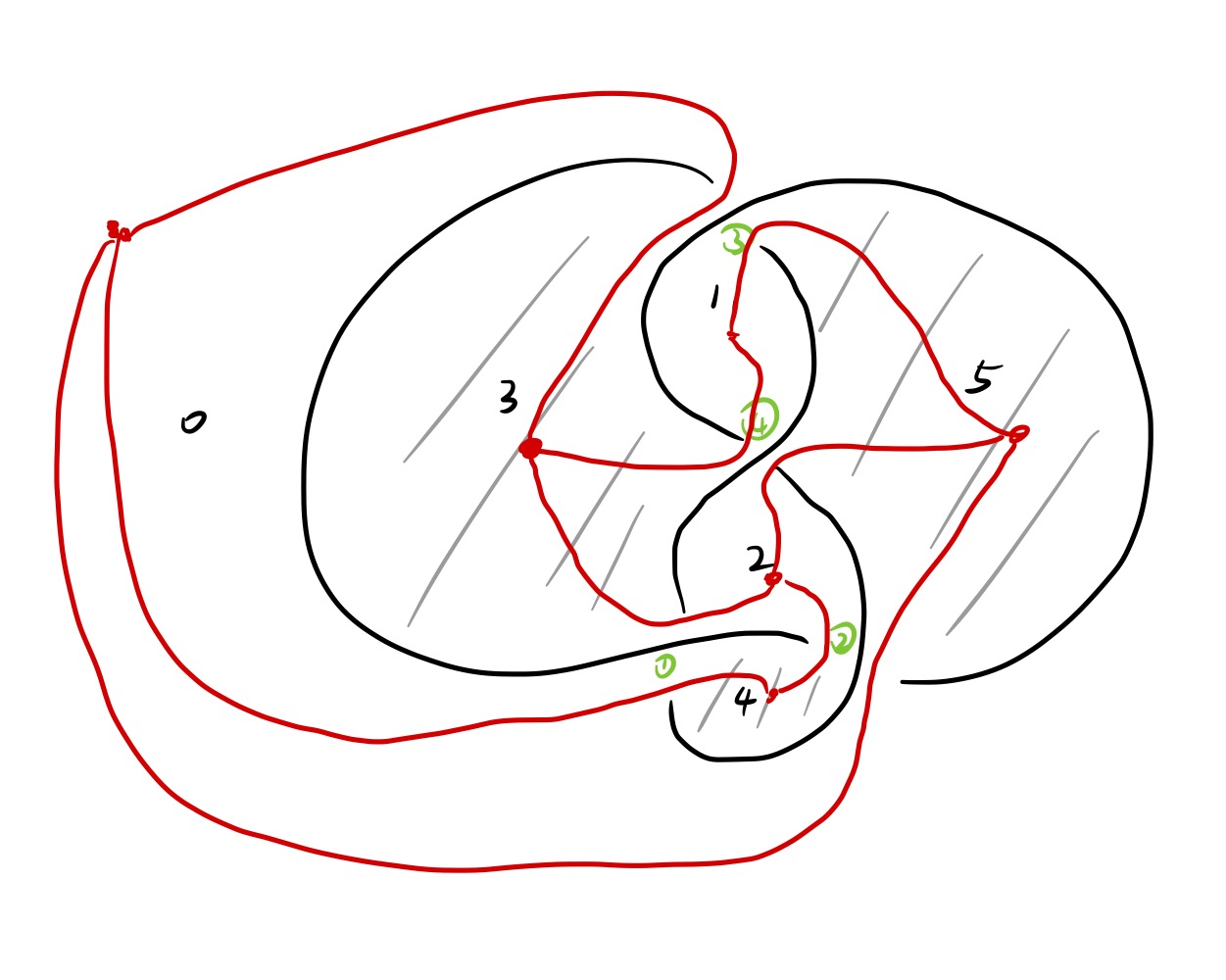}
	\caption{The link of the top vertex of the Dehn complex associated to the figure-eight knot diagram}
	\label{Fig:linkfig8}
\end{figure}

\subsubsection{Coned Dehn complex}
The construction of the generalised Dehn complex will be useful for later proofs.

\begin{proposition}[Coned Dehn complex]\label{Prop:Dehn}
	The coned Dehn complex $\mathcal{D}$ of a diagram on a surface is a retraction of the cube complex $r: \mathcal{C}\to \mathcal{D}$ where on each cube we restrict to a square on $\partial N(L)$.
\end{proposition}
\begin{proof}
	See Figure~\ref{Fig:link_oct} for the retraction to Dehn complex.
\end{proof}

Note that from Construction~\ref{const: octa_decomp}, item (6), we have:
\begin{observation}
	The coned Dehn complex is isomorphic to the complex we obtain from the link complement in thickened surface if we cone the top and bottom surfaces to vertices and run through the same process as in Construction~\ref{const:Dehn}.
\end{observation}

As we can see from item (6) in Construction~\ref{const: octa_decomp}, the surface has a cell structure given by the dual diagram $\pi(L)^*$.

\begin{proposition}
	The generalised Dehn complex is isomorphic to the coned Dehn complex.
\end{proposition}

From now on, we just abbreviate and call both the \emph{Dehn complex}, when the context of classical or virtual is clear.

\subsubsection{Nonpositive curvature}
With the above setups, we are ready to prove that under mild conditions, our cube complex $\mathcal{C}$ is nonpositively curved.
\begin{notation}
	Let $x\in \mathcal{D}\subset \mathcal{C}$ be a vertex in the cell decomposition, and let $lk(x)$ and $Link(x)$ denote the link of $x$ in the Dehn complex $\mathcal{D}$ and the Aitchison complex $\mathcal{C}$, respectively.
\end{notation}

\begin{definition}[edge representativity]
	Let $\pi(L)$ be a diagram of a link $L$ on an orientable surface $S$ in a compact, orientable, irreducible 3-manifold $M$. The \emph{edge representativity} $e(\pi(L),S)$ is the minimum number of intersections between the diagram $\pi(L)$ with any essential curve $l\subset S$.
\end{definition}

\begin{proposition}\label{prop:npc}[First half of Theorem~\ref{Thm:npc}]
	Let $L$ be an alternating link in $M=S\times I$ with a reduced alternating diagram $D$ on $S$. Suppose further that $e(\pi(L),S)\geq 4$.  Let $\mathcal{C}$ be the cube complex associated to $D$ in $M^*\setminus N(L)$. Then $\mathcal{C}$ and its double $\hat{\mathcal{C}}$ across $\partial E(L)$ are nonpositively curved.
\end{proposition}

To prove this proposition, we need a few lemmas.

\begin{lemma}
	The link of the vertex $lk(P_+)$ inside the Dehn complex is part of $Link(P_+)$ of the vertex $P_+$ inside $\mathcal{C}$. Similarly for $P_-$. 
\end{lemma}
\begin{proof}
	This can be seen from Construction~\ref{const: octa_decomp} of the Dehn complex. Since the Dehn complex $\mathcal{D}$ is the retraction of each octahedron to its middle level, it follows that the link $lk(P_+)$ is part of $Link(P_+)$.
\end{proof}

\begin{lemma}\label{Lem:subdivide}
	The link $Link(P_+)$ of the cube complex $\mathcal{C}$ can be obtained by subdividing the dual structure $\pi(L)^*$ of the surface $S$. Similarly for $Link(P_-)$.
\end{lemma}

\begin{proof}
	From Construction~\ref{const: octa_decomp}, we can see that the Dehn complex is dual to the crossing arcs. Suppose we take the dual structure $\pi(L)^*$. Add one vertex at each crossing, or the center of each square in $\pi(L)^*$. Add one edge from each crossing vertex to each corner vertices of the square. See Figure~\ref{Fig:linkedges} for the subdivision.
	
	Similarly for $P_-$.
\end{proof}

\begin{figure}
	\centering
	\includegraphics[width=0.4\linewidth]{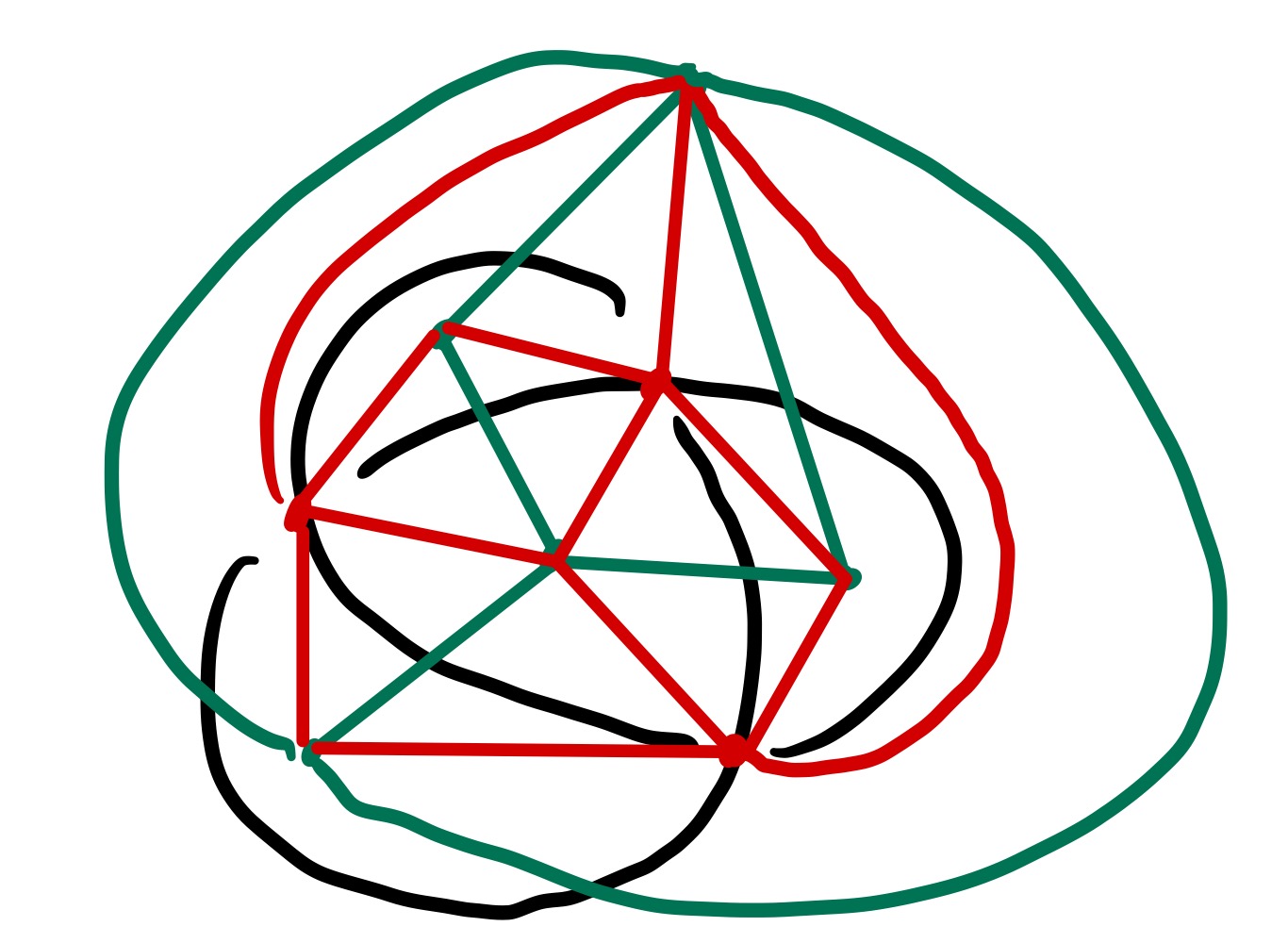}
	\caption{A schematic picture of the subdivision of the link of a vertex in the classical Aitchison complex. Green and red edges in $Link(P_+)$. Green represents the edges coming from $lk(P_+)$, and red represents the further subdivision.}
	\label{Fig:linkedges}
\end{figure}

\begin{proof}[Proof of Proposition~\ref{prop:npc}]
	The link of vertices in $\hat{\mathcal{C}}$ can be divided into two types, one tiles the link, called $Link(P_0)$, and the other tiles the boundary surfaces where the vertex corresponds to the coning point, called $Link(P_+)$ or $Link(P_-)$. Observe that $\hat{\mathcal{C}}$ induces a cubing of $\partial N(L)$, such that each vertex of this cubing are incident to 4 edges and 4 squares. Therefore, the link $Link(P_0)$ in $\hat{\mathcal{C}}$ is made up of 4 spherical triangles, which glue to become a unit hemisphere. These types of links are all flags in Gromov's condition.
	
	For the other type of link, e.g., the link $Link(P_+)$, we consider the diagram $D$ and its dual $D^*$. We can see that $D^*$ gives a disc decomposition of $S$, isomorphic to $lk(P_+)$. By Lemma~\ref{Lem:subdivide}, we can further subdivide to get $Link(P_+)$, and the subdivision consists of edges joining a crossing and a region point. We color $lk(P_+)$ to be green and the further subdivision to be red. See Figure~\ref{Fig:linkedges} for the subdivision and red-green coloring. The only possibility for the link $Link(P_+)$ to violate Gromov's flag condition is from the edges coming from $lk(P_+)$, shown as green in Figure~\ref{Fig:linkedges}. Since we have edge representativity $e(\pi(L),S)\geq 4$, this can be avoided because otherwise there will exist an essential curve that intersect the diagram less than 4 times.
	
	Therefore, after examining all possibilities for links for vertices in the cube complex $\hat{\mathcal{C}}$, the cube complex $\mathcal{C}$ is nonpositively curved.
\end{proof}

Links of different types of vertices can be drawn as Figure 6 and Figure 7 in Kim-Kim-Yoon~\cite{kim2018octahedral}. The link $Link(P_+)$ and $Link(P_-)$ together form the boundary of the thickened surface $S\times \{\pm 1\}$, which satisfies Proposition~\ref{prop:flag}.

\begin{corollary}[Second half of Theorem~\ref{Thm:npc}]\label{cor:ess}
	All edges of the cube complex $\mathcal{C}$ are essential.
\end{corollary}
\begin{proof}
	Since edges in the decomposition are all edges of the nonpositively curved cube complex $\mathcal{C}$, they are all local geodesics, and they are also not null-homotopic when they are lifted to the dual complex $\hat{\mathcal{C}}$.
\end{proof}

Notice that this proof is much less wordy than the one from Sakuma-Yokota~\cite{sakuma2018application}, since our vertices go to totally geodesic surfaces. Sakuma-Yokota have to cap off the upper and lower 2-spheres by balls, and therefore take into consideration all the paths in the decomposition, whereas we only have to look at the edges.

\section{From octahedra to tetrahedra}\label{sec:tet}

\begin{theorem}\label{thm:4term}
	The link complement in a thickened surface $S\times I \setminus L$ admits a triangulation $\mathcal{T}$ by tetrahedra, each of which has two ideal vertices and two truncated vertices.
\end{theorem}
\begin{proof}
	The proof follows exactly from the construction from Kim-Kim-Yoon~\cite{kim2018octahedral}. Cut along the crossing arc inside each octahedron in the partially truncated octahedral decomposition of $S\times I\setminus L$. We can obtain 4 partially truncated tetrahedra, all with 2 ideal vertices and 2 truncated vertices. See Figure~\ref{Fig:4term} for more information of the cutting. If the link complement is decomposed into $c$ octahedra, then it is automatically decomposed into $4c$ tetrahedra.
\end{proof}
\begin{figure}
	\centering
	\includegraphics[width=0.6\linewidth]{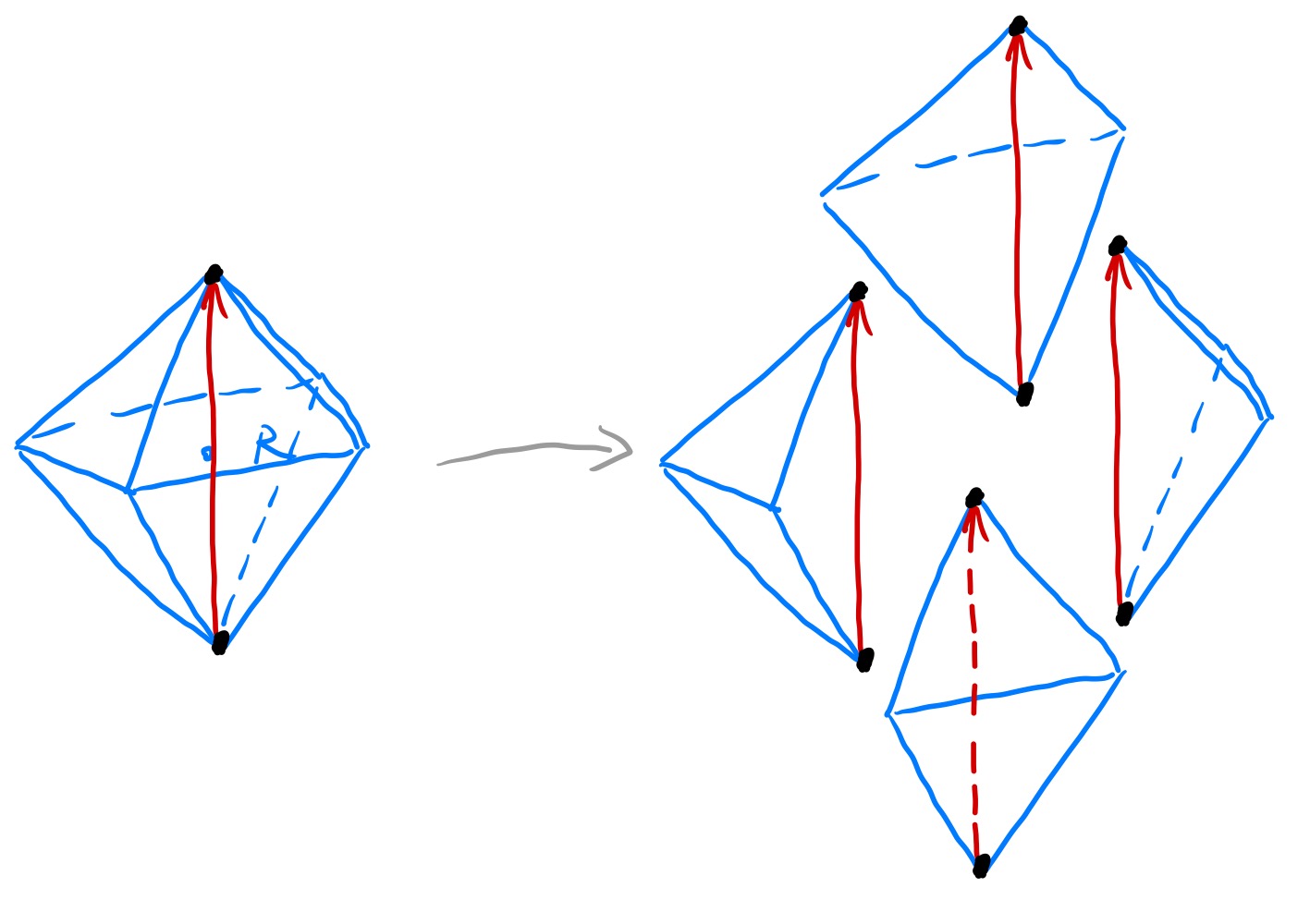}
	\caption{The four term triangulation coming from the octahedral decomposition.}
	\label{Fig:4term}
\end{figure}
\begin{corollary}
	All edges of the triangulation $\mathcal{T}$ are essential when $S\times I\setminus L$ is hyperbolic. 
\end{corollary}
\begin{proof}
	Corollary~\ref{cor:ess} implies that all edges of the cube are essential. These are all edges of the triangulation except for the added edges, which are crossing arcs. However, the crossing arcs are essential because they lie on a quasifuchsian surface, by work of Howie-Purcell~\cite{HowiePurcell}.
\end{proof}

\subsubsection{Relationship with stellated triangulation or bipyramid decomposition}

The octahedral decomposition and its resulting triangulation are the same as the bipyramid triangulation used by Adams et al~\cite{MR4046919} from their 2020 work, where they assign angles to the bipyramids and study the volume of links in thickened surfaces. We give a brief introduction of how they are related here. We also recommend the thesis by Kaplan-Kelly~\cite{kaplan2024right}, where she gives a detailed description of the stellated triangulation. See Figure~\ref{Fig:stellated} for a schematic picture of the bipyramid decomposition of link complements in thickened surfaces.

\begin{figure}
	\centering
	\includegraphics[width=0.6\linewidth]{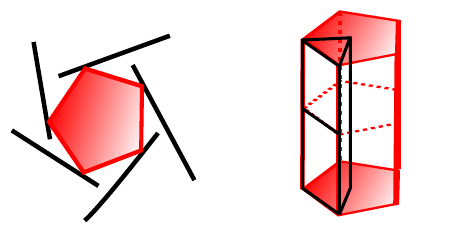}
	\caption{For each region, we have the number of edges equal the number of tetrahedra. These tetrahedra form a bipyramid in the stellated triangulation of link complement in thickened surface.}
	\label{Fig:stellated}
\end{figure}

Each truncated octahedron can be thought of as union of 4 truncated tetrahedra as in Figure~\ref{Fig:decomp}, top, and they all intersect in the crossing arc.

Or alternatively, we can directly start with 4 prism shaped objects for each crossing, and deform each prism to be a partially truncated tetrahedra. See Figure~\ref{Fig:deform_to_tet} for the deformation. Note that Weeks~\cite{weeks2005computation} also used this deformation for computation of the hyperbolic structure for classical knot complements. See also Figure~\ref{Fig:2Ddeform} for the two-dimensional picture for the deformation.
\begin{figure}
	\centering
	\includegraphics[width=0.4\linewidth]{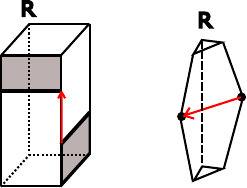}
	\caption{Weeks' deformation to get a triangulation for link complement}
	\label{Fig:deform_to_tet}
\end{figure}

\begin{figure}
	\centering
	\includegraphics[width=0.5\linewidth]{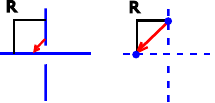}
	\caption{Two dimensional view of the deformation to triangulation}
	\label{Fig:2Ddeform}
\end{figure}


\subsection{The nonorientable case} \label{subsec:nonoct}
As the chunk decomposition for nonorientable case, we can also obtain an octahedral decomposition of a link complement in a twisted $I$-bundle over a nonorientable surface.

\begin{observation}
	The construction is nearly identical to Construction~\ref{const: octa_decomp}, but now for a diagram on a nonorientable surface $S$. We take the link exterior inside a twisted $I$-bundle $S\widetilde{\times}I$:
	\begin{enumerate}
		\item We have one octahedron for each crossing.
		\item We drag all four of the middle four vertices to $\widetilde{S}=\bdy S\widetilde{\times}I$, where opposite pairs of middle vertices are pulled in the same direction (up or down), in a neighborhood of the crossing.
		\item Each triangular face is glued to a neighboring one, same as in Figure~\ref{fig:gluing}.
		\item The order of gluings is determined by following the boundary curve of each region. See Figure~\ref{Fig:stellatednon}.
	\end{enumerate}
\end{observation}

\begin{figure}
	\centering
	\includegraphics{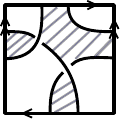}
	\caption{The boundary of a region determines the order of gluing of the octahedra}
	\label{Fig:stellatednon}
\end{figure}

See Figure~\ref{fig:octnonorientable} for a schematic picture of the octahedral decomposition of the link $2_6$ from Matveev-Nabeeva~\cite{MatveevNabeeva} in $K\widetilde{\times} I$

\begin{figure}
	\centering
	\includegraphics[width=0.5\linewidth]{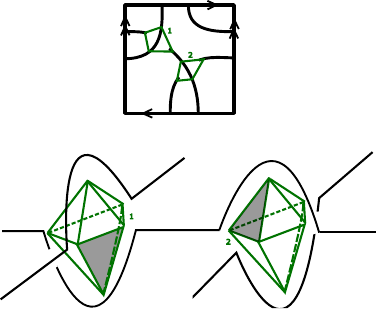}
	\caption{The octahedral decomposition from a simple diagram on a Klein bottle.}
	\label{fig:octnonorientable}
\end{figure}

\begin{theorem}
	The link exterior inside $S\widetilde{\times} I$ admits a truncated octahedral decomposition, and hence a stellated triangulation and an associated cube complex.
\end{theorem}

\begin{lemma}
	Let $\widetilde{S}$ be the orientable double cover of a nonorientable surface $S$, and $\widetilde{L}$ the double cover of a link $L\subset S\widetilde{\times} I$. If $\widetilde{S} \times I \setminus \widetilde{L}$ is nonpositively curved, then $S\widetilde{\times} I\setminus L$ is nonpositively curved.
\end{lemma}
\begin{proof}
	By definition of nonpositive curvature, $\widetilde{S} \times I \setminus \widetilde{L}$ has a $CAT(0)$ space as its universal cover. Then since $\widetilde{S} \times I \setminus \widetilde{L}$ doubly covers $S\times I\setminus L$ by~\cite[Proposition~2.18]{purcell2024alternating}, the universal cover of $S\times I\setminus L$ is also $CAT(0)$.
\end{proof}

\section{Further thoughts and open questions}\label{sec:open}

\subsection{Pseudo hyperbolic structure}

The octahedral decomposition and its resulting triangulation can be lifted to hyperbolic 3-space to obtain a fundamental domain and the group action that shows the gluings. It would be interesting to generalise work of Kim-Kim-Yoon~\cite{kim2018octahedral} to obtain explicit info on representations and hyperbolic structure. But for $S\times I$ with high genus $S$, it is harder to calculate the hyperbolic structure since we have totally geodesic boundary components.

In Sakuma-Yokota's work~\cite{sakuma2018application}, they obtain a geometric solution associated to the triangulation coming from the octahedral decomposition, yet their technique only works for classical alternating diagrams. On the other hand, if we follow Kim-Kim-Yoon~\cite{kim2018octahedral} to put labels on each edge of the octahedral decomposition, we obtain a system of nonlinear equations. However, we need to add extra edges coming from the totally geodesic components in order to generalise.

\subsection{The interaction of chunk decomposition and octahedral decomposition}
The checkerboard surfaces can be seen inside each octahedron, as in Construction~\ref{const: octa_decomp}. From Purcell-Su~\cite{purcell2024alternating}, we can see that the checkerboard decomposition is done by cutting through both checkerboard surfaces to obtain top and bottom polyhedra for classical case and chunks for the generalised case. Following Sakai-Sakuma~\cite{sakai2024two}, one can see inside each octahedron how the checkerboard surfaces cut through the entire space and how the pieces are decorated by the knot diagram.

Therefore, even if we have chunks (homeomorphic to $S\times I$) instead of polyhedra, we can still see the pieces inside each octahedron and how they are glued together to form the chunks, only now we are looking ``sideways". Half of them will glue together to give one chunk, and the other half the other, or one chunk for nonorientable case. See Figure~\ref{Fig:chunkincube} for the pieces of chunks inside each octahedron. Note that it is the same picture for classical case and the case of links in thickened surfaces, orientable or not.

\begin{figure}
	\centering
	\includegraphics{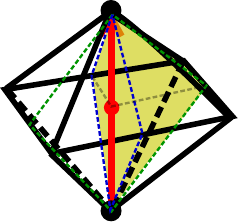}
	\caption{Pieces of chunk inside each octahedron. The yellow shaded part denote one piece cut by the checkerboard hyperplanes inside each octahedron.}
	\label{Fig:chunkincube}
\end{figure}

Since we already have a nonpositive curvature theorem, we can almost run the same process as Sakai-Sakuma~\cite{sakai2024two} by cutting the cubes into pieces and part or all of them will glue to become chunks. After further cutting our cube complex $\mathcal{C}$, we can lift the cubes to the $CAT(0)$ universal cover. The difference is, after coning, we have extra information on the middle vertices. This is another piece of work that can be done.

\newpage
\bibliographystyle{apalike}
\bibliography{ref}

\end{document}